\documentclass[12pt,a4paper]{amsart} 
\usepackage{amsmath}
\usepackage{amsthm}
\usepackage{amssymb}
\usepackage{ot1patch}
\usepackage{fancyhdr}
\newtheorem{thm}{Theorem}[section]
\newtheorem{lem}[thm]{Lemma}
\newtheorem{cor}[thm]{Corollary}
\newtheorem{prop}[thm]{Proposition}

\theoremstyle{remark}
\newtheorem{rem}[thm]{Remark}
\newtheorem*{rem*}{Remark}

\theoremstyle{definition}
\newtheorem{dfn}[thm]{Definition}
\newtheorem{ex}[thm]{Example}

\numberwithin{equation}{section}

\newcommand{\Rz}{\mathbb{R}}

\begin{document}
\title{Linear programming on non-compact polytopes and the Kuratowski convergence with application in economics}
\author{Anna Denkowska, Maciej Denkowski \and Marta Kornafel}
\begin{abstract}
The aims of this article are two-fold. First, we give a geometric characterization of the optimal basic solutions of the general linear programming problem (no compactness assumptions) and provide a simple, self-contained proof of it together with an economical interpretation. Then, we turn to considering a dynamic version of the linear programming problem in that we consider the Kuratowski convergence of polyhedra and study the behaviour of optimal solutions. Our methods are purely geometric. 
\end{abstract}
\subjclass{}
\keywords{Linear programming, tangent cone, normal cone, Kuratowski convergence}
\date{December 20th 2015, Revised: January 2nd 2017}

\maketitle

\section{Introduction}

A classical problem in optimization theory and one that has a wide range of applications economics, is the linear programming problem (LP for short). In the canonical form it is written as:
$$\begin{cases}
c^Tx\to\min\\
Ax=b\\
x\geq 0,
\end{cases}$$
where $c\in\mathbb{R}^n$ is the cost vector, $c^T$ is its transposed (thus $c^Tx=\langle c,x\rangle$ denotes the usual inner product), $A$ is the matrix of a linear function $A\colon \mathbb{R}^n\to \mathbb{R}^m$, $b\in\mathbb{R}^m$, $x\geq 0$ means $x_i\geq 0$ for $i=1,\dots, n$, and it is usually assumed that the set of feasible solutions $F_{A,b}:=\{x\in{\Rz}^n\mid Ax=b, x\geq 0\}$ is \textit{compact}, so that a solution necessarily exists.

The classical solution to this PL problem is given by the so called \textit{simplex method}. Observe that even a discrete LP problem, that is  one in which we consider $F^d_{A,b}:=F_{A,b}\cap \mathbb{Z}^n$ can be reduced to the above one by considering the LP problem on the convex hull $\mathrm{conv}(F^d_{A,b})$. It is a classical and easy to show fact that the solutions to the LP problem lie all on the boundary $\partial F_{A,b}$ (more accurately: on the relative boundary computed in the unique affine space of the lowest possible dimension containing $F_{A,b}$) and it is sufficient to look for them among the \textit{extremal points} of $F_{A,b}$. Recall that given a closed, convex set $F\subset{\Rz}^n$, a point $x_0\in F$ is called \textit{extremal} --- we write then $x_0\in F^*$ --- if 
$$
\exists x_1, x_2\in F, \exists t\in (0,1)\colon x_0=(1-t) x_1+tx_2\Rightarrow x_1=x_2.
$$

The following fact is well-known:
\begin{prop}
Assuming that the rank $\mathrm{rk} A=m<n$ (which is not really restrictive), a point $x\in F_{A,b}$ is extremal if and only if it is a basic feasible solution.
\end{prop}
Of course, a basic feasible solution is a point $x\in F_{A,b}$ such that either $x=0$, or the columns of $A$ corresponding to the non-zero coordinates of $x$ are linearly independent.

\noindent\textbf{Notation.} Given an $m\times n$ matrix $A$ we denote by $A^{(i_1,\dots, i_k)}$ the matrix $A$ without the rows with indices $\neq i_j$. On the other hand $A^j$ will denote the $j$-th column of $A$. Finally, we write $A=(A_1,\dots, A_m)$ with $A_i\colon {\Rz}^n\to {\Rz}$ that are linear forms.

Let us stress that we will use interchangeably the words {\it linear polytope} and {\it polyhedron} meaning actually {\it convex polyhedron} in the following sense:

\begin{dfn}
A nonempty set $E\subset {\Rz}^n$ is called a {\it convex polyhedron} or just polyhedron, if there is a non-zero linear mapping $A\colon {\Rz}^n\to{\Rz}^m$ and a vector $b\in {\Rz}^m$ such that $E=\{x\in{\Rz}^n\mid Ax\leq b\}$. 
\end{dfn}
Observe that this definition excludes ${\Rz}^n$ and that a polyhedron need not be compact (\footnote{A compact convex polyhedron is usually called a {\it polytope}.})

For a point $x\in E$,we denote by $J(x)=\{i\in \{1,\dots, m\}\mid A_ix=b_i\}$ the set of {\it active constraints at} $x$.

Of course, the describing linear mapping $A$ is not uniquely determined, unless we require it to be minimal in the following sense. Let $d$ be the dimension of the convex polyhedron $E$. Then there is an affine $d$-dimensional subspace $V\subset{\Rz}^n$ containing $E$ (the affine hull or envelope of $E$, denoted also by $
\operatorname{Aff}(E)$) and such that $E=\overline{\operatorname{int}_V E}$. This affine hull is described by $n-d$ equations $\langle w_j,x\rangle =u_j$. Now, let $f_k(E)$ denote the number of $k$-dimensional faces of $E$. In particular, $f_0(E)=\#E^*$ is the number of {\it vertices} or extremal points, whereas $f_{d-1}(E)$ is the number of {\it facets} (faces of maximal possible dimension) (\footnote{Note that $f_0(E)$ may be zero, unlike $f_{d-1}(E)$.}). Then in $V\equiv {\Rz}^d$ we need exactly $f_{d-1}(E)$ linear inequalities $A_ix\leq b_i$ to describe $E$, as this set is the intersection of as much half-spaces as it has facets. Therefore, a minimal description of $E$ is given by $n-d$ linear equations together with $f_{d-1}(E)$ linear inequalities.

Hereafter we will deal with the general linear programming problem:
$$\begin{cases}
c^Tx\to\min\\
Ax\leq b\\
\end{cases}\leqno{(GLP)}$$
with $A\colon {\Rz}^n\to{\Rz}^m$ linear with $m\geq n$. This is somehow motivated by the following proposition, that we prove for the convenience of the reader.

\begin{prop}\label{prop}
Let $E_{A,b}=\{x\in\mathbb{R}^n\mid Ax\leq b\}$ with $A$ as above. Then $\bar{x}\in E_{A,b}^*$ implies that $m\geq n$ and there are indices $i_1<\ldots<i_n$ such that $A^{(i_1,\dots, i_n)}\bar{x}=(b_{i_1},\dots, b_{i_n})$ and $\det A^{(i_1,\dots, i_n)}\neq 0$. In particular, 
$$
\bigcap_{i\in J(\bar{x})}\{x\in{\Rz}^n\mid A_ix=b_i\}=\{\bar{x}\}
$$
where $J(\bar{x})=\{i\in\{1,\dots, q\}\mid A_ix=b_i\}$ are the indices of the active constraints at $\bar{x}$. 
\end{prop}
The point $\bar{x}\in E_{A,b}^*$ is called a \textit{vertex} of the polytope $E_{A,b}$. In the usual terminology $\bar{x}$ is called a \textit{basic optimal solution}.

\begin{proof}[Proof of Proposition \ref{prop}] 
The point $\bar{x}$ being extremal, it cannot lie in the interior of $E_{A,b}$. Thus, there is an index $i\in J(\bar{x})$. We may assume that $i=1$. Now, we use the fact that for linear subspaces $V,W\subset{\Rz}^n$ we have $\dim V\cap W\geq \dim V+\dim W-n$.

The case $n=1$ being obvious, we may assume that $n\geq 2$. Moreover, no harm will be inflicted on generality, if we assume that $A_j\not\equiv 0$. Thus $\dim\mathrm{Ker} A_j=n-1$ for all $j$. 

 Had we $A_i\bar{x}<b_i$ for all $i>2$, we would find a ball $B$ centred at $\bar{x}$ and such that $B\cap A_1^{-1}(b_1)\subset E_{A,b}$. But this set has dimension $n-1>0$ and so $\bar{x}$ is not extremal. Therefore there is $i>1$ in $J(\bar{x})$. We may assume that $i=2$. Since $\dim A_1^{-1}(b_1)\cap A_2^{-1}(b_2)\geq n-2$, we conclude that this has to be an equality for some index $i>1$ (otherwise $\bar{x}$ would not be extremal). Then we may repeat the preceding argument in order to conclude that either there must be an index $i>2$ in $J(\bar{x})$, or $n=2$ and we have the equality sought for. It is then clear that the procedure must end and that $\bar{x}$ would not be extremal if we needed less than $n$ steps. Hence we have $A_1,\dots, A_n$ such that $\bigcap_{i=1}^n A_i^{-1}(b_i)=\{\bar{x}\}$. This in turn implies that $\bigcap_{i=1}^n \mathrm{Ker}A_i=\{0\}$ which means that $A_1,\dots, A_n$ are linearly independent which ends the proof.
\end{proof}

\begin{rem}
In this article we \textsl{do not assume} that $E_{A,b}$ is compact. Note that in real life we often do not know exactly \textsl{all} the constraints (we lack data) of a given engineering or economics problem and actually we are dealing with a non-compact $E_{A,b}$. 

Note that the interest in matters conerning linear programming is still quite important (see e.g. \cite{DGLL}). Our approach is very basic, nevertheless it gives some applicable results.

We have two aims: to explain under which condition the GLP problem is solvable and give a geometric solution to it, and to study what happens when we approximate the polyhedron $E_{A,b}$ by similar polyhedra, in particular --- how do the solutions behave.
\end{rem}


\section{Solving the GLP problem using normal cones}

For a given set $E\subset{\Rz}^n$ and a point $a\in\overline{E\setminus\{a\}}$ we define the usual \textit{Peano tangent cone} of $E$ at $a$ as the cone
$$
C_a(E)=\{v\in{\Rz}^n\mid \exists E\ni x_\nu\to a, \lambda_\nu>0\colon \lambda_\nu(x_\nu-a)\to v\},
$$
and the \textit{normal cone} of $E$ at $a$ as the cone
$$
N_a(E)=\{w\in{\Rz}^n\mid \forall v\in C_a(E), \langle v,w\rangle\leq 0\},
$$
which means that any vector $w\in N_a(E)$ forms with any vector $v\in C_a(E)$ an angle greater than or equal to $\pi/2$.

Keeping the notations introduced so far we obtain first:

\begin{lem}\label{stozek styczny}
Let $w\in E_{A,b}$. Then $$C_w(E_{A,b})=\bigcap_{i\in J(w)}\{x\in{\Rz}^n\mid A_i x\leq 0\}.$$
\end{lem}
\begin{proof}
Both sets contain the origin. Take a non-zero vector $v$ from the tangent cone. Let $E_{A,b}\ni x_\nu\to w$ and $\lambda_\nu>0$ be the sequences yielding $\lambda_\nu(x_\nu-w)\to v$. For $i\in J(w)$ we have $$A_i(\lambda_\nu(x_\nu-w))=\lambda_\nu (A_ix_\nu-b_i)\leq 0,$$
for $\lambda_\nu$ are positive. Therefore, $A_i$ being continuous, we obtain $A_iv\leq 0$, as required.

Take now $v\neq 0$ belonging to the set on the right-hand side. Then for $i\in J(w)$ we have $A_iw=b_i$ and so for any $\varepsilon>0$, we obtain 
$$
A_i\left({\varepsilon}v+ w\right)=\varepsilon A_i v+b_i\leq b_i.
$$
If in turn $i\notin I(w)$, then $A_i w<b_i$, and so suitably small $\varepsilon$ ensure that
$$
A_i\left({\varepsilon}v+ w\right)= \varepsilon A_i v+A_i w<b_i
$$
still holds. Now, taking $\varepsilon_\nu$ decreasing to zero and $\lambda_\nu:=\frac{1}{\varepsilon_\nu}$ we conclude that $x_\nu:=\varepsilon_\nu v+w\in E_{A,b}$ and $\lambda_\nu(x_\nu-w)=v$.
\end{proof}
In particular, we can reconstruct $E_{A,b}$ from its vertices:
\begin{prop}\label{odtwarzanie}
If $E_{A,b}^*\neq\varnothing$, then $$E_{A,b}=\bigcap_{w\in E_{A,b}^*}(C_w(E_{A,b})+w).$$
\end{prop}

\begin{proof}
The inclusion `$\subset$' is obvious (compare with the previous proof). Take now a point $x$ from the set on the right-hand side. Then for any $i\in\bigcup_{w\in E_{A,b}^*}J(w)=:J$ we obtain $A_i (x-w)\leq 0$, i.e. $A_ix\leq b_i$. It remains to observe that if there is an index $j\in \{1,2,\dots, m\}\setminus J$, then the correspponding inequality $A_jx\leq b_j$ is superfluous in the description of $E_{A,b}$. We may thus conclude that $x\in E_{A,b}$.
\end{proof}
It follows also from the lemma above that for $w\in E_{A,b}$, $N_w(E_{A,b})=\{\sum_{i\in J(w)} \lambda_i a_i\mid \lambda_i\geq 0, i\in J(w)\}$ where $A_i(x)=\langle a_i,x\rangle$. Therefore, we easily obtain the following remark.
\begin{cor}
The polyhedron $E_{A,b}$ is unbounded iff either $E_{A,b}^*=\varnothing$, or $E_{A,b}^*\neq\varnothing$ and
$$
\bigcup_{w\in E_{A,b}^*} N_w(E_{A,b})\neq {\Rz}^n.
$$
\end{cor}

Now we are ready to prove in an elementary fashion the following basic theorem:
\begin{thm}\label{N}
Let a linear mapping $A\colon {\Rz}^n\to {\Rz}^m$ of rank $n$ define a (possibly unbounded) polyhedron $E_{A,b}$. Then the functional $f(x)=c^Tx$ attains its minimum on $E_{A,b}$, if and only if 
$$
-c\in\bigcup_{w\in E_{A.b}^*}N_w(E_{A,b}).
$$
In particular, the minimum is attained at those vertices $w\in E_{A,b}$ for which $-c\in N_w(E_{A,b})$.
\end{thm}
Note that such a result can of course be deduced from some much more general results in convex analysis involving subgradients and so on (compare e.g. \cite{ChZ}). In our opinion, however, it is rather useful -- in view of the importance of linear programming -- to have a straightforward and self-contained proof, based on simple geometric notions. 

Before proving the theorem, we note the following lemma:
\begin{lem}
Let $f(x)=c^Tx$ and consider a nonempty closed set $C\subset {\Rz}^n$. Let $V$ be the affine envelope of $C$. Then $f$ attains $\inf_{x\in C} f(x)$ iff there is a point in the relative boundary $x_0\in \partial_V C$ for which $f(x_0)=\inf_{x\in C} f(x)$.
\end{lem}
\begin{proof}[Proof of Theorem \ref{N}]
Using the previous lemma it is easy to check the well-known fact  that $f$ attains its minimum on $E_{A,b}$ iff there is a vertex $w\in E_{A,b}^*$ such that $f(w)=\inf_{E_{A,b}} f$. 

Therefore, it suffices to prove that for a given vertex $w$, $\langle c,x\rangle\geq \langle c,w\rangle$ for all $x\in E_{A,b}$ iff $-c\in N_w(E_{A,b})$, i.e. $\langle -c, v\rangle \leq 0$ for all $v\in C_w(E_{A,b})$. 

We begin with the `if' part. By Proposition \ref{odtwarzanie}, for any $x\in E_{A,b}$ we have $x-w\in C_w(E_{A,b})$. Now, $\langle c,w\rangle \leq \langle c,x\rangle $ is equivalent to $\langle c, w-x\rangle\leq 0$, or in other words $\langle -c, x-w\rangle \leq 0$. The latter we know to be true.

Now, for the `only if' part, to prove that $-c\in N_w(E_{A,b})$ we take any point $v\in C_w(E_{A,b})$. Then we consider the approximating sequence $\lambda_\nu(x_\nu-w)\to v$ with $E_{A,b}\ni x_\nu\to w$ and $\lambda_\nu>0$. We have $\langle c, x_\nu\rangle\geq \langle c, w\rangle$ and this remains true when we multiply both sides by $\lambda_\nu$, whence $\langle c, \lambda_\nu(x_\nu-w)\rangle\geq 0$. After multiplying both sides by $-1$ and passing to the limit we obtain $\langle -c, v\rangle\leq 0$ as required.
\end{proof}

\begin{rem}
The theorem above has a straightforward real-life application. It says that any functional $c_p^Tx$, where $p$ is a parameter, attains its minimum (or maximum --- by duality) at a\textsl{ fixed} vertex $w\in E_{A,b}^*$ as long as the cost vectors $-c_p$ remain in the normal cone $N_w(E_{A,b})$. Moreover, we do not need the compactness of $E_{A,b}$ to obtain this, which means that some of the constraints are negligible. 

For instance, suppose that a factory produces $n$ products selling them at prices $c_j$ that could vary, as dictated by the market, in the intervals $[a_j, b_j]$  ($j=1,\dots, n$) and the constraints $Ax\leq b$ correspond to how the machines can be set up (and the data may be incomplete, as they are in real life, i.e. $E_{A,b}$ can be non-compact). Assuming that the set up $\bar{x}\in E_{A,b}$ is optimal for the profit $c^Tx$ to be maximal, the theorem says precisely how may the prices evolve without raising the need of changing the set up $\bar{x}$ in order to keep the profit maximal: to do this we only need to compute the normal cone at $\bar{x}$ (or more accurately, at the vertex corresponding to this optimal point).
\end{rem}

\section{Kuratowski convergence and LP problem}

First, let us recall the notion of convergence of sets we will be using. We will state the definition for a natural type of nets (generalized sequences). Consider a set $E\subset{\Rz}^k_t\times{\Rz}^n_x$ and denote by $E_t:=\{x\in{\Rz}^n\mid (t,x)\in E\}$ its section at $t\in{\Rz}^k$. Also, let $\pi(t,x)=t$ be the natural projection and fix $t_0\in\overline{\pi(E)}$.

\begin{dfn} We write $x\in \limsup_{t\to t_0} E_t$ iff for any neighbourhood $U\ni x$ and for any neighbourhood $V\ni t_0$ there exists a point $t\in V\cap \pi(E)$ different from $t_0$ and such that $E_t\cap U\neq\varnothing$. We call the resulting set \textit{the Kuratowski upper limit} of $E_t$ at $t_0$.

We write $x\in\liminf E_t$ iff for any neighbourhood $U\ni x$ there is a neighbourhood $V\ni t_0$ such that for all $t\in V\cap \pi(E)\setminus\{t_0\}$, we have $E_t\cap V\neq\varnothing$. We call the resulting set \textit{the Kuratowski lower limit} of $E_t$ at $t_0$. 

We say that $E_t$ \textit{converges} to the set $F\subset{\Rz}^n$ iff $$\limsup_{t\to t_0} E_t=\liminf_{t\to t_0} E_t=F.$$ We write then $F=\lim_{t\to t_0} E_t$ or $E_t\stackrel{K}{\longrightarrow} F$ ($t\to t_0$).
\end{dfn}

\begin{rem}
Of course, $\liminf_{t\to t_0} E_t\subset\limsup_{t\to t_0} E_t$ and both sets are closed. Moreover, they do not change, if we take $\overline{E_t}$ instead of $E_t$. Therefore, it is natural to restrict ourselves only to closed sets. Observe also that 
$$
F=\lim_{t\to t_0} E_t\ \Longleftrightarrow\ \limsup_{t\to t_0} E_t\subset F\subset \liminf_{t\to t_0} E_t.
$$
\end{rem}
Note that a sequence of sets $(E_\nu)$ can be identified with the $t$-sections of the set $E=\bigcup_\nu \{1/\nu\}\times E_\nu\subset{\Rz}\times{\Rz}^n$ and thus the upper and lower limits of $(E_\nu)$ for $\nu\to+\infty$ may be understood as $\limsup_{t\to 0} E_t$ and $\liminf_{t\to 0} E_t$, respectively. In this case it is easy to see that $\liminf E_\nu$ consists of all the possible limits of converging sequences $x_\nu\in E_\nu$, while $\limsup E_\nu$ consists of all the possible limits of converging subsequences $x_{\nu_s}\in E_{\nu_s}$.

\begin{rem}
For compact sets, the Kuratowski convergence is exactly the convergence in the usual Hausdorff measure. Note also that for a given set $E\subset{\Rz}^n$ and $a\in\overline{E}$ we have
$$
C_a(E)=\limsup_{\varepsilon \to 0} \frac{E-a}{\varepsilon}.
$$
\end{rem}
We will denote by $H(a;b)$ the affine hypersurface $\langle a,x\rangle =b$, where $||a||=1$ and by $\hat{H}(a;b)$ the half-space defined by $\langle a,x\rangle\leq b$. Observe that $\hat{H}(a_\nu;b_\nu)\stackrel{K}{\longrightarrow}\hat{H}(a;b)$, whenever $a_\nu\to a$, $b_\nu\to b$ and the same is true for the corresponding hypersurfaces.

Recall that we say that two sets $E_1,E_2\subset{\Rz}^n$ {\it can be separated}, if there are $a,b$ such that $E_i\subset \hat{H}((-1)^ia;(-1)^ib)$,  $i=1,2$ which for convex sets is equivalent to $0$ not being an interior point of $E_1-E_2$ (cf. \cite{R} Theorem 2.39).

Let us also note the following easy Proposition:
\begin{prop}
The Kuratowski limit of a converging sequence of convex set is a convex set.
\end{prop}
\begin{proof}
Let $C_\nu$ be convex sets converging to a set $C_0$. Take $x,y\in C_0$. Then, due to the convergence, these points are limits of some sequences $x_\nu,y_\nu\in C_\nu$, respectively. But $[x_\nu,y_\nu]\subset C_\nu$ and clearly, the limit of a sequence of segments is a segment (maybe reduced to a point). It follows easily that $[x,y]\subset C_0$.
\end{proof}

We start this section with a short discussion of the following question:\\
Assume that $\varnothing\neq C\subset{\Rz}^n$ is a closed, convex set and $f\colon C\to {\Rz}$ a continuous function with $M:=\sup_{x\in C} f(x)<+\infty$. When does there exist a point $x_0\in C$ such that $f(x_0)=M$?

Of course, the question makes sense in particular for an unbounded set $C$. In general there is not much hope to obtain a positive answer: for $n=1$ and $C=[0,+\infty)$ take $f(x)=\arctan x$. If $f$ were linear, we would have a realizing point in this case.

Even though $f$ is linear, such a realizing point $x_0$  may not exist in general, unless $C$ is a polyhedron. Take $n=2$, $C=\{(x,y)\mid x>0, y\geq 1/x\}$ and $f(x,y)=-y$. 

Nevertheless, the following is true:
\begin{prop}
Let $C=\{x\in{\Rz}^n\mid \langle a_i,x\rangle\leq b_i, i=1,\dots,k\}$ be a nonempty polyhedron and $f(x)=\langle c,x\rangle$ with $M:=\sup_{x\in C} f(x)<+\infty$. Then, independently of the fact whether $C$ is bounded or not, there is a point $x_0\in C$ such that $f(x_0)=M$.
\end{prop}
\begin{proof}
If $f\not\equiv 0$, we may assume that $||c||=1$ and $C$ is unbounded. Then we have $C\subset \hat{H}(c;M)$ and clearly $\mathrm{dist}(f^{-1}(M),C)=0$. Take a sequence $(x_\nu)\subset C$ for which $f(x_\nu)\to M$. Each point $x_\nu$ can be written as $\frac{f(x_\nu)}{||c||^2}c+z_\nu=f(x_\nu)c+z_\nu$ where $z_\nu\in\mathrm{Ker} f$. This gives us points $Mc+z_\nu\in f^{-1}(M)$ and $y_\nu\in C$ realizing their distance to $C$. Then it is easily seen that $f(y_\nu)\to M$. 

We may assume now that $||y_\nu||\to+\infty$ (otherweise the limit of a convergent subsequence yields a point in $C$ realizing $M$ for $f$). Since $y_\nu\in\partial C$, then passing to a subsequence we may assume furhter that $\langle a_i,y_\nu\rangle=b_i$ for $i=1,\dots,N$ with $N\geq 1$, while  $\langle a_i,y_\nu\rangle<b_i$ for $i=N+1,\dots,k$. Then choosing a subsequence we will get $\frac{y_\nu-y_1}{||y_\nu-y_1||}\to v$ and of course $[y_1,y_\nu]\subset C$ for each $\nu$. Then $\ell:=y_1+\mathbb{R}_+v\subset C$ and we obtain $\mathrm{dist}(\ell,f^{-1}(M))=0$, i.e. $\ell\subset f^{-1}(M)$.
\end{proof}

Suppose that $S_\nu$ is the set of  solutions of $c^Tx\to\min$ on $E_{A,b_\nu}$. When do these sets converge to the set of solutions of $c^Tx\to \min$ on $E_{A,b}$ and what can guarantee that the latter is nonempty?

The main theorem of the preceding section gives a possible answer to this problem. Namely, if we know how do behave the normal cones and if we know that the cost vectors are `nicely' related to them, then we can say that the limit problem has a solution and even give the vertex realizing it.

\begin{thm}\label{main}
Let $E_\nu\subset {\Rz}^n$ be a sequence of convex polyhedra such that $E_\nu\stackrel{K}{\longrightarrow} E\neq\varnothing$ where $\varnothing\subsetneq E\subsetneq{\Rz}^n$, and one of the following conditions is satisfied: either $E$ is compact and there is a uniform bound $\#E_\nu^*\leq M$, or there is a uniform bound $f_{\dim E_\nu-1}(E_\nu)\leq M$. Then \begin{enumerate}
\item $E$ is a convex polyhedron, too, and $\#E^*\leq \#E_\nu^*$, for almost all indices;
\item For any vertex $v\in E^*$ there is a sequence of vertices $E_\nu^*\ni v_\nu\to v$ and $C_{v_\nu}(E_\nu)\stackrel{K}{\longrightarrow} C_v(E)$, as well as $N_{v_\nu}(E_\nu)\stackrel{K}{\longrightarrow} N_v(E)$;
\item If $f_\nu\colon {\Rz}^n\to{\Rz}$ is a sequence of linear forms converging to $f\colon{\Rz}^n\to{\Rz}$ and such that each $f_\nu$ attains its maximum on $E_\nu$, then $f$ attains its maximum on $E$; moreover, $\operatorname{argmax} f_\nu\stackrel{K}{\longrightarrow}\operatorname{argmax} f$, provided one of the following conditions holds: either $E$ is compact, or $\#E^*=\#E_\nu^*$ for indices large enough, or $\max_{E} f$ exists and is the limit of $\max_{E_\nu} f_\nu$.
\end{enumerate}
\end{thm}
\begin{rem}
In the noncompact case a uniform bound on the number of vertices is in general not enough to obtain a polyhedron as the limit. Consider an approximation of the unit circle in ${\Rz}^2$ by $\nu$-gones inscribed in it. Embed the plane ${\Rz}^2\times\{0\}\to {\Rz}^3$ and consider the infinite cones spanned over the $\nu$-gones from the vertex at $(0,0,1)$ --- these are the sets $E_\nu$. Of course, they are convex, non-compact polyhedra converging to the regular cone spanned over the circle from the point $(0,0,1)$. Not only the limit is no longer a polyhedron, but it has infinitely many extremal points, while $E_\nu^*=\{(0,0,1)\}$.

Of course, the assumption that $\varnothing\subsetneq E\subsetneq{\Rz}^n$ is unavoidable, too, cf. $(-\infty,\nu]\stackrel{K}{\longrightarrow} {\Rz}$, while $(-\infty,-\nu]\stackrel{K}{\longrightarrow}\varnothing$ --- in both cases the limit is not a polytope according to our definition.

Finally, the last point can be illustrated by the following example in ${\Rz}^2$: let $f_\nu(x,y)=f(x,y)=-y$ and let $E_\nu=\{(x,y)\in{\Rz}^2\mid x,y\geq 0, \nu y\geq \nu-x\}$. Then $E_\nu$ converges to $E=\{(x,y)\mid x\geq 0,y\geq 1\}$, but the maximizers do not converge.
\end{rem}

In the course of the proof we shall be using the following notions.
\begin{dfn}
Two linear inequalities $\langle a_i,x\rangle \leq b_i$ with $||a_i||=1$, $i=1,2$ are called {\it inverse equivalent} ({\it i-e} for short), if $a_1=-a_2$ and $b_1=-b_2$. 
\end{dfn}
Put together, two i-e inequalities describe the affine hypersurface $H(a_1;b_1)=H_2(a_2,b_2)$.

Let $a_1,a_2\in{\Rz}^n$ be non-colinear unit vectors. We put $\displaystyle v(a_1,a_2):=\frac{a_1+a_2}{|||a_1+a_2||}$.

\begin{lem}\label{stozki normalne}
Let $V_\nu$ and $V$ be real cones (\footnote{I.e. $tV\subset V$ for any $t\geq 0$.}) in ${\Rz}^n$ with $V_\nu\stackrel{K}{\longrightarrow} V$. Then the normal cones $N(V_\nu)$ converge to $N(V)$.
\end{lem}
\begin{proof}
Take $w\in \limsup N(V_\nu)$  and $v\in V$. Then there is a sequence $V_\nu\ni v_\nu\to v$ and a subsequence $N(V_{\nu_k})\ni w_{\nu_k}\to w$. Since $\langle w_{\nu_k}, v_{\nu_k}\rangle\leq 0$, we get $\langle w,v\rangle \leq 0$, i.e. $w\in N(V)$.

Fix now $w\in N(V)$. Without loss of generality we may assume that $||w||=1$. Then $V\subset \hat{H}(w;0)$ and the type of convergence implies that for large indices, $V_\nu\subset \hat{H}(w,0)$. Indeed, $V_\nu\cap {\Rz}^n\setminus\{0\}=V_\nu\setminus\{0\}$ converge to $V\setminus\{0\}$, whence $V_\nu\setminus\{0\}\cap \mathrm{int} \hat{H}(w,0)$ converge to $V\setminus\{0\}\cap\mathrm{int} \hat{H}(w;0)=V\setminus\{0\}$. It follows that $w\in N(V_\nu)$, for almost all indices, i.e. $N(V)\subset\liminf N(V_\nu)$.
\end{proof}

\begin{proof}[Proof of Theorem \ref{main}]
If $E$ is compact, then so are the sets $E_\nu$, from some index onward (this follows directly from the definition of the convergence, compare e.g. \cite{DD}). Then it is easy to see that  $f_{k}(E_\nu)\leq \binom{f_0(E_\nu)}{k+1}$, since a $k$-dimensional face must contain $k+1$ affinely independent points that define it. Therefore, we will be working under the assumption that the number of facets is uniformly bounded. 

By passing to a subsequence, we may assume that all the polyhedra $E_\nu$ have the same dimension $d$ and then that the numbers $f_{k}(E_\nu)$, $k=0,
\dots,d-1$ are independent of the index, both in the compact and non-compact case. 

What is more, we may assume that $d=n$ due to the following argument. Let $V_\nu=\operatorname{Aff}(E_\nu)$ and let $\vec{V}_\nu$ be the underlying vector space. Then by the Zarankiewicz Theorem (i.e. sequential compacity), after passing to a subsequence we can find a limit $\vec{V}_0=\lim\vec{V}_\nu$ which is, obviously, also a $d$-dimensional vector space. But if we fix a point $x_0\in E$ and take any sequence $E_\nu\ni x_\nu\to x_0$, then we see that $V_\nu$ converge to $V_0:=\vec{V}_0+x_0$ and of course, $V_0\supset E$. 

It follows now easily from the definition of the Kuratowski convergence that we may assume that all the sets $E_\nu$ lie in the same $d$-dimensional space $V_0$, or rather that, actually, we are dealing with $n$-dimensional polyhedra. 

This implies that we can describe the sets $E_\nu$ in the following manner:
$$
E_\nu\colon \langle a_{i,\nu},x\rangle\leq b_{i,\nu},\> i=1,\dots, N=f_{n-1}(E_\nu),
$$
with $||a_{i,\nu}||=1$ for all $i,\nu$. Again, passing to a subsequence, we may assume that $a_{i,\nu}\to a_i$ for each $i$ when $\nu\to +\infty$. 

Now, each sequence $(b_{i,\nu})_\nu$ may be bounded or unbounded. Note that since $E\neq\varnothing$, we cannot have $b_{i,\nu}\to -\infty$. On the other hand, if $b_{i,\nu}\to +\infty$, then from the set-theoretical point of view, the corresponding $i$-th constraint stops playing any role in the description, i.e. we may forget it in the limit. The only interesting case is when (for a subsequence) $b_{i,\nu}\to b_i\in\mathbb{R}$.

Assume that, passing to a subsequence, $(b_{i,\nu})_\nu$ have limits $b_i$ for $i=1,\dots, N'$ and diverge to $+\infty$ for $i=N'+1,\dots, N$. Observe that there must be $N'\geq 1$, because $E\neq{\Rz}^n$. Consider first the set
$$
E'\colon \langle a_i,x\rangle\leq b_{i,\nu},\> i=1,\dots, N'.
$$
It may happen that some pairs of the constraints above are i-e. Suppose that this is the case for the indices $i,j$. It may happen that $a_{i,\nu}=-a_{j,\nu}$ for all indices (but, of course, $b_{i,\nu}\neq b_{j,\nu}$ due to the assumption that $\dim E_\nu=n$) --- we will say then that the pair of constraints $(i,j)$ is {\it parallel i-e}. Suppose, however, that it is not the case, i.e. we can assume that $a_{i,\nu},a_{j,\nu}$ are not colinear, for all indices (as usual, by extracting a subsequence). Then $v_{ij,\nu}=v(a_{i,\nu}, a_{j,\nu})$ make sense and due to the type of convergence, the positive cones ${\Rz}_+a_{i,\nu}+{\Rz}_+a_{j,\nu}$ must converge to an affine half-plane. Therefore, the vectors $v_{ij,\nu}$ have a well-defined limit $v_{ij}$ (for once there is no need to extract a subsequence). 

In this situation, adding to the description of $E_\nu$ the inequality $\langle v_{ij,\nu},x\rangle\leq u_{ij,\nu}$ where $u_{ij,\nu}$ is the value at a point $x_0$ satisfying $\langle a_{i,\nu}, x_0\rangle = b_{i,\nu}$ and $\langle a_{j,\nu}, x_0\rangle = b_{j,\nu}$ (\footnote{There must necessarily exist such a point for large indices $\nu$, for by assumptions $H(a_{i,\nu};b_{i,\nu})$ and $H(a_{j,\nu};b_{j,\nu})$ are not parallel.}), does not change $E_\nu$. We may assume that $v_{ij,\nu}\to v_{ij}$.

We face again two possibilities. Namely $W_{ij,\nu}:=H(a_{i,\nu},b_{i,\nu})\cap H(a_{j,\nu},b_{j,\nu})$ may converge (after passing to a subsequence) to an affine $n-2$-dimensional subspace $W_{ij}$, or to the empty set: this depends on whether the translating vectors $w_{ij,\nu}$ in $W_{ij,\nu}=w_{ij,\nu}+\vec{W}_{ij,\nu}$ with $||w_{ij,\nu}||=\mathrm{dist}(0,W_{ij,\nu})$ have a bounded subsequence or not. Clearly, this corresponds to the behaviour of $u_{ij,\nu}$, i.e. we will obtain $W_{ij}$, provided the $u_{ij,\nu}$ converge to some $u_{ij}\in{\Rz}$. Otherwise, if $W_{ij,\nu}\stackrel{K}{\longrightarrow}\varnothing$, then $\hat{H}(a_{i,\nu},b_{i,\nu})\cap \hat{H}(a_{j,\nu},b_{j,\nu})$ converge to an affine hyperplane and we do not need to bother adding the additional constraint $\langle v_{ij,\nu},x\rangle\leq u_{ij,\nu}$ to the description of $E_\nu$, as it does not play any role in the limit.

We introduce now the set $$E'':=E'\cap\{x\in{\Rz}^n\mid\langle v_{ij},x\rangle\leq u_{ij}, (i,j)\in \mathcal{I}\}$$ where $\mathcal{I}$ is the set of all pairs of indices from $\{1,\dots,N'\}$ that are i-e but not parallel i-e and for which $u_{ij}$ is well-defined. We claim that $E=E''$.

It is obvious that $E\subset E''$: for $E=\liminf E_\nu$, whence any $x_0\in E$ is the limit of some sequence of points $x_\nu\in E_\nu$ and we just pass to the limit in the description (\footnote{Remember that we are working on a subsequence of $E_\nu$ chosen by taking into account $\mathcal{I}$, among other conditions.}). To prove the converse, take a point $x_0\in E''$. There is $x_0\in E'$ and if we had only strict inequalities in the description, we would be able to move $a_i$ and $b_i$ to $a_{i,\nu}$ and $b_{i,\nu}$, for sufficiently large indices $\nu$, without changing the inequalities; i.e. $x_0\in E$ in such a case. Assume, however, that there is
\begin{align*}
&\langle a_i,x_0\rangle=b_i,\> i=1,\dots, N''\\
&\langle a_i,x_0\rangle<b_i,\> i=N''+1,\dots, N',
\end{align*}
where $1\leq N''\leq N'$. We may also assume that for $\nu \gg 1$, 
\begin{align*}
&\langle a_{i,\nu},x_0\rangle>b_{i,\nu} \> i=1,\dots, N'''\\
&\langle a_{i,\nu},x_0\rangle\leq b_{i,\nu},\> i=N'''+1,\dots, N',
\end{align*}
for some $1\leq N'''\leq N''$. Observe that it implies that for the distance $d_\nu:=\mathrm{dist}(x_0,E_\nu)$ which is realized by exactly one point $x_\nu\in E_\nu$ (due to the convexity of the sets $E_\nu$), we necessarily have $\langle a_{i_\nu,\nu},x_\nu\rangle =b_{i_\nu,\nu}$, for some $i_\nu$thatnecessarily belongs to $\{1,\dots, N'''\}$ (the point realizing the distance has to lie on the boundary). Then we may assume that $i_\nu=:i_0$ does not depend on $\nu$, i.e., to be more specific, that we have (possibly after a permutation of $\{1,\dots,N'''\}$)
\begin{align*}
&\langle a_{i,\nu},x_\nu\rangle=b_{i,\nu},\> i=1,\dots, i_0\\
&\langle a_{i,\nu},x_\nu\rangle<b_{i,\nu},\> i=i_0+1,\dots, N'.
\end{align*}

Now, $d_\nu\to d:=\mathrm{dist}(x_0,E)$, because, if $\varepsilon>0$, then $\mathbb{B}(x_0,d-\varepsilon)\cap E=\varnothing$, while $\mathbb{B}(x_0,d+\varepsilon)\cap E\neq \varnothing$ and these conditions hold also for $\nu\gg 1$, due to the convergence (cf. \cite{DD} Lemma 2.1). This implies $d-\varepsilon<d_\nu<d+\varepsilon$, $\nu\gg 1$, as required. Moreover, $(x_\nu)$ has to be a bounded sequence, since $d_\nu=||x_0-x_\nu||$, so that we may assume that $x_\nu\to \bar{x}_0$. Of course, $\bar{x}_0\in E$ and it realizes $d$. This realizing point is unique, because $E$ is a convex set, too.

Suppose that all the points $x_\nu$ lie on a facet of the corresponding set $E_\nu$, i.e. $i_0=1$. Then, there must be $x_0=x_\nu+d_\nu a_{i_0,\nu}$, which means that $\langle a_{i_0,\nu},x_\nu\rangle =b_{i_0,\nu}$ yields 
$$
\langle a_{i_0,\nu},x_0\rangle-d_\nu =b_{i_0,\nu}.
$$
By passing to the limit, we get 
$$
\langle a_{i_0},x_0\rangle-d =b_{i_0}.
$$
But $\langle a_{i_0},x_0\rangle=b_{i_0}$, whence $d=0$, i.e. $x_0=\bar{x}_0\in E$.

Suppose that $i_0>1$ and let $x_{\nu,i}$ denote the orthogonal projections of $x_0$ onto $H(a_{i,\nu};b_{i,\nu})$ for $i=1,\dots, i_0$ and $d_{\nu, i}=\mathrm{dist}(x_0, H(a_{i,\nu};b_{i,\nu}))$. By the argument above, $d_{\nu,i}\to 0$. Now, if $\{1,\dots, i_0\}\times\{1,\dots, i_0\}\cap \mathcal{I}=\varnothing$, then this implies that $x_\nu\to x_0$, i.e. $x_0=\bar{x}_0$. Otherwise, let us consider the corresponding additional constraints $\langle v_{ij,\nu},x\rangle\leq u_{ij,\nu}$ together with the orthogonal projections $x_{\nu, ij}$ to $H(v_{ij,\nu}, w_{ij,\nu})$ and the corresponding distances $d_{\nu,ij}$. Note that we necessarily have $\langle v_{ij,\nu},x_\nu\rangle=u_{ij,\nu}$, whence, as earlier, we obtain $d_{\nu,ij}\to 0$. Nowarguing similarly asin the proof of \cite{BS} Theorem 1.1 based on \cite{LM} Formula (13) (compare \cite{BS} Theorem 1.3;in particular the constant in this theorem is bounded), this is sufficient to conclude that $x_\nu\to x_0$ (\footnote{Essentially, what is taken care of here may be illustrated by the following simple example in the plane: let $E_\nu$ be given by $-y\leq 0$ and $y-(1/\nu)x\leq 0$ which are i-e constraints; these sets converge to $E=[0,+\infty)\times \{0\}$ but $E'$ is the whole $x$-axis.}). This ends the proof of (1).

Once we have obtained (1) with the convergence of the facets, we directly get (2) from simple linear algebra (compare Proposition \ref{prop}): if $E$ is $n$-dimensional, then a vertex is described by $n$ linearly independent inequalities. Then thenearby inequalities are linearly independent and it follows that they define a vertex approaching the one in question. If, however, we had some i-e inequalities so that $\dim E=k<n$, then the same kind of argument works for $k$ describing functions restricted to $\mathrm{Aff}(E)$. If we take into account also the i-e inequalities, then we see that the vertex must be a limit of vertices. Proposition \ref{stozek styczny} implies now the convergence of the tangent cones, while Lemma \ref{stozki normalne} yields the assertion concerning the normal cones. Finally, the first part of (3) holds, because $f(x)=\langle c,x\rangle$ attains a maximum on $E$ iff $\hat{H}(c/||c||,b)\supset E$ for some $b$; since this holds foreach index $\nu$, it will hold also in the limit. For the second part, the compactness of $E$ implies the compactness of $E_\nu$ and Theorem \ref{N} gives the result. The same argument is valid, if the number of vertices is constant, since the maximum is realized in a vertex. If we know that the maxima $M_\nu$ converge to the maximum $M$ of $f$ on $E$, then we easily get the convergence of the maximizers $H(c_\nu/||c_\nu||;M_\nu/||c_\nu||)\cap E_\nu$ to $H(c/||c||,M/||c||)\cap E$ using the half-spaces (compare \cite{R} Theorem 4.32).
\end{proof}
\begin{rem}
Let us observe that a particular case of this theorem can be directly derived from \cite{R} Theorem 4.32 (a). Essentially, this theorem states that if $A_\nu\to A$ for linear maps $A_\nu,A\colon{\Rz}^n\to{\Rz}^m$ and the sets $A({\Rz}^n)$ and $\prod_{i=1}^m(-\infty,b_i]$ cannot be separated, then $E_{A_\nu,b_\nu}\stackrel{K}{\longrightarrow}E_{A,b}$ when $b_\nu\to b$. This, however, does not cover entirely our result.
\end{rem}

The last point of the Theorem is a particular instance of the De Giorgi-Franzoni Theorem, namely:
\begin{thm}
Assume that the vectors $c_\nu\in{\Rz}^n$ converge to $c$ and let $M_\nu$ denote the set of minimizers of $f_\nu(x)=c_\nu^Tx$ in $E=E_{A,b}$. Then $M_\nu$ converge in the sense of Kuratowski to the set $M\subset E$ being the set of minimizers for the limiting functional $f(x)=c^Tx$.
\end{thm}

It follows from the proof of Theorem \ref{main} that for linear polytopes we have also the following strong result.
\begin{thm}
A sequence of linear polytopes $E_\nu$ converges iff their boundaries $\partial E_\nu$ converge and then
$$
\partial\lim_{\nu\to+\infty} E_\nu=\lim_{\nu\to+\infty} \partial E_\nu.
$$
Moreover, if the polytopes $E_\nu$ have nonempty interiors, then  $\mathbb{R}^n\setminus E_\nu$ converges to the complement of the limit of the sets $E_\nu$.
\end{thm}
\begin{rem}
Of course, this type of result necessarily requires at least a convexity assumption. Indeed, if $K$ is the unit disc in the plane, then  $\displaystyle\overline{K\setminus\frac{1}{\nu} K}$ converges to $\overline{K}$ but the boundaries do not converge to the boundary of the limit.
\end{rem}

\section{Examples of application}
We end our paper with some simple examples of application. Let us start with an economical one that illustrates Theorem \ref{N}. 
\begin{ex}
A factory produces $n$ articles that are sold at prices $c_1,\dots, c_n$ per unit. Of course, the prices are subject to some variations. We denote by $a_ij$ the coefficient encoding how much of the $j$-th raw material is used to produce the $i$-th article. Let $\beta_j$ be an upper bound for the stock of the $j$-th raw material. 

As is well-known, in order to maximize the profit we have to solve a linear programming problem given by 
$$
\begin{cases}c^Tx\to \max\\
[a_{ij}]x\leq \beta\\
x\geq 0,
\end{cases}
$$
where $c=(c_1,\dots, c_n)$ is the cost vector and $x=(x_1,\dots,x_n)$ gives the number of articles produced.By passing to the dual problem, we may rewrite this as $-c^Tx\to \min$. Here 
$$A=\left[\begin{array}{ccc}a_{11}&\dots &a_{1n}\\ 
\dots &\ddots &\dots\\
a_{m1} &\dots &a_{mn}\\
-1&\dots & 0\\
\dots &\ddots &\dots\\
0&\dots &-1\end{array}\right]$$
and $b=(\beta, 0,\dots, 0)^T$.

Ifi $\bar{x}$ denotes an optimal point, it implies a certain regulation of the machines in the factory. Now, Theorem \ref{N} tells us that this regulation is optimal (gives a maximal profit) as long as the prices represented by the cost vector $c$ do not leave the cone $N_a(E)$ for an appropriate choice of the vertex $a$ in the feasible set $E$. Note that computing the cones from initial data is an easy task. 
\end{ex}

\begin{ex}\label{Example 2.2}    
Consider the producer's system \cite{Debreu}, where the production set $Y$ is given by constraints $Ay\leq b$ with the properties $Y^*\neq \emptyset$ and $Y\cap(-Y)\subset\{0\}$. The goal is to maximize the producer's profit $p^Ty$. Therefore taking $f(y)=-p^Ty$ we look for minimum over the set $Y$. By theorem 2.4 the optimal production plan is at some $y^*\in Y^*$ and the optimal price is the one satisfying $p\in N_{y^*}(Y)$.

In the next picture we present this example in two-dimensional space of goods. Pay attention that we do not have any returns to scale, i.e. in contrast to \cite{Debreu} we waive the assumption about the convexity of production set $Y$. Additionally, the constants below satisfy $b>a>0$, $b>1$(\footnote{This assumption is only technical. The fact that $a,b>0$ implies that in the considered example the production set meets the standard economic expectations. Thanks to the fact that $b>a$ and $b>1$ it is possible to determine the optimal production plan}).

$$\begin{array}{rcl}
Y=Y_1\cup Y_2&=&
\{(y_1,y_2)\in\mathbb{R}^2: y_2\leq a \wedge y_2\leq -\frac 12 y_1 \wedge y_2\leq -2y_1 \}\cup\\
&&\{(y_1,y_2)\in\mathbb{R}^2:  y_2\leq -b \wedge y_2\leq -2y_1+2\}
\end{array}
$$
Then $Y^*=\{ (-2a,a), (0,0), (b,-2b), (b+1,-2b) \}$ and:
$$
\begin{array}{l}
N_{(-2a,a)}(Y)=\{(p_1,p_2)\in\mathbb{R}^2:
p_2\geq 0 \wedge p_2\geq 2p_1
\}\\

N_{(0,0)}(Y)=\{(p_1,p_2)\in\mathbb{R}^2:
2p_1\geq p_2 \geq \frac 12 p_1
\}\\

N_{(b,-2b)}(Y)=\emptyset\\

N_{(b+1,-2b)}(Y)=\{ (p_1,p_2)\in\mathbb{R}^2:
p_2\geq 0 \wedge p_2\geq \frac 12 p_1
\}
\end{array}
$$

For the prices from the corresponding cones the profit from production is:

$$
\begin{array}{l}
\pi_{(-2a,a)}(p_1,p_2)=-2ap_1+ap_2=:\pi_1\\
\pi_{(0,0)}(p_1,p_2)=0\\
\pi_{(b+1,-2b)}(p_1,p_2)=(b+1)p_1-2b p_2=:\pi_2\\
\end{array}
$$
Moreover, the constraint $b>1$ implies $\pi_1>\pi_2$. Therefore the optimal producion plan is  $y^*=\{(-2a,a) \}$ giving the maximal profit $\pi^*=\pi_1$.
\end{ex}

\subsection{Kuratowski convergence and LP problem}

\begin{ex}\label{Example 3.1}
Continuing the \ref{Example 2.2}, consider the producer's system, in which the producer is introducing some innovations. The innovations may be understood as the employment of some new technologies into the production process, rearrengement of the existing production process in the way that increases production abilities, etc. All of them result in extension of the set of possible production plans, denoted as $Y_\nu$. We naturally ask about the influence of those changes on the optimal plans. When can we assure that realisation of a current producer's optima leads to the optimal production in the final set $Y$? 
The positive answer is given by Theorem \ref{main}, provided the sets $Y_\nu$ converge to the set $Y$ in Kuratowski sense.

To illustrate the example let's consider again the following numerical example in two-dimensional space of goods. As before, the constants below satisfy $b>a>0$, $b>1$.

$$\begin{array}{rcl}
Y_\nu=Y_{1,\nu}\cup Y_{2,\nu}&=&
\{(y_1,y_2)\in\mathbb{R}^2: y_2\leq a \wedge y_2\leq -\frac \nu 2 y_1 \wedge y_2\leq -\frac 2\nu y_1 \}\cup\\
&&\{(y_1,y_2)\in\mathbb{R}^2:  y_2\leq -b \wedge y_2\leq -\frac 2\nu y_1+2\}
\end{array}
$$

Then the Kuratowski limit of the sequence $(Y_\nu)$ when $\nu\nearrow 1$ is the set $Y$ defined in the example 2.2. Moreover, the sequence $(Y_\nu)$ is ascending, i.e. for $\mu>\nu$ it holds $Y_\nu\subset Y_\mu$. This represents the described expansion of production set.

The candidates for $\nu$-optimal production plans are 
$Y^*_\nu=\{ (-\frac{2a}{\nu},a),\ (0,0),\ (\nu b,-2b),\linebreak
(\nu b+1,-2b) \}$, while the corresponding normal cones are:
$$
\begin{array}{l}
N_{(-\frac{2a}{\nu},a)}(Y)=\{(p_1,p_2)\in\mathbb{R}^2:
p_2\geq 0 \wedge p_2\geq \frac{2}{\nu} p_1
\}\\

N_{(0,0)}(Y)=\{(p_1,p_2)\in\mathbb{R}^2:
\frac{2}{\nu} p_1\geq p_2 \geq \frac{\nu}{2} p_1
\}\\

N_{(\nu b,-2b)}(Y)=\emptyset\\

N_{(\nu b+1,-2b)}(Y)=\{ (p_1,p_2)\in\mathbb{R}^2:
p_2\geq 0 \wedge p_2\geq \frac \nu 2 p_1
\}
\end{array}
$$

The profits generated by the production plans and price vectors from corresponding normal cones are:

$$
\begin{array}{l}
\pi_{(-\frac{2a}{\nu},a)}(p_1,p_2)=-\frac{2a}{\nu} p_1+ap_2=:\pi_{1,\nu}\\
\pi_{(0,0)}(p_1,p_2)=0\\
\pi_{(\nu b+1,-2b)}(p_1,p_2)=(\nu b+1)p_1-2b p_2=:\pi_{2,\nu}\\
\end{array}
$$

By similar arguments as before $y^*_\nu=(-\frac{2a}{\nu},a)$ and $\pi^*_\nu=\pi_{1,\nu}$. Clearly, $\lim\limits_{\nu\nearrow 1} y^*_\nu=y^*$ and $\lim\limits_{\nu\nearrow 1}\pi^*_\nu=\pi^*$.\\
\end{ex}

Now we present an example in which the number of vertices is reduced in the limit passing.

\begin{ex}\label{Example 3.2}
Consider the descending sequence of production sets:
$$Y_\nu=\{(y_1,y_2)\in\mathbb{R}^2: y_2\leq a\ \wedge\ y_2\leq -\nu y_1-2a\ \wedge\ y_2\leq -\frac 1\nu y_1-\frac{(\nu+1)^2}{\nu}\cdot a\ \wedge\ p_1\leq 0
\}$$
with $\nu\nearrow 1$. Then for any $1>\nu>0$ the set of veritices is 
$$Y^*_\nu=\left\{(-\frac{3a}{\nu},a), (-a,-(2+\nu)a), \left(0,-\frac{(\nu+1)^2}{\nu}\cdot a \right) \right\}.$$
The Kuratowski limit of the sequence $(Y_\nu)$ when $\nu\nearrow 1$ is the set
$$Y=\{ (y_1,y_2)\in\mathbb{R}^2: y_2\leq a\ \wedge\ y_2\leq -y_1-2a\ \wedge\ p_1\leq 0 \},$$
for which $Y^*=\{ (-3a,a), (0,-4a) \}.$ Clearly, for any $\nu\in(0,1)$ the optimal plans are $y^*_\nu=(-\frac{3a}{\nu},a)$, which converges to the optimal production plan in the limiting set $y^*=(-3a,a)\in Y$. 
\end{ex}

\bigskip
\noindent{\small\textsc{Addresses:}}\\{\tiny
(A.D. \and M. K.)\hfill (M.D.)\\
Cracow University of Economics\hfill Jagiellonian University\\
Department of Mathematics\hfill Faculty of Mathematics and Computer Science\\
 Rakowicka 27\hfill Institute of Mathematics\\
 31-510 Cracow, Poland\hfill\L ojasiewicza 6\\
 {\tt anna.denkowska@uek.krakow.pl} \hfill 30-348 Krak\'ow, Poland\\
{\tt marta.kornafel@uek.krakow.pl}\hfill
{\tt maciej.denkowski@uj.edu.pl}
}


\begin{thebibliography}{DGLL}
\bibitem{ChZ} E. Chong, S. \.Zak, {\it An Instroduction to Optimization}, Wiley Eds 2004;

\bibitem{BS} C. Bergthaller, I. Singer, {\it The distance to a polyhedron}, Linear Alg. Appl. 169 (1992), 111-129;

\bibitem{DM} G. Dal Maso, {\it Introduction to $\Gamma$-convergence}, Birkh\"auser 1991;


\bibitem{Debreu} G. Debreu, {\it Theory of value}, Yale University Press 1959;

\bibitem{DD} Z. Denkowska, M. Denkowski, {\it The Kuratowski convergence and connected components}, J. Math. Anal. Appl. 387 (2012), 48-65;

\bibitem{DGLL} A. Daniilidis, M. Goberna, M. Lopeza, R. Luchetti, {\it Lower semicontinuity of the feasible set mapping of linear systems relative to their domains}, preprint 2014

\bibitem{LM} P.-J. Laurent, B. Martinet, {\it M\'ethodes duales pour le calcul du minimum d'une fonction convexe sur une intersection de convexes} in {\it Symposium on Optimization, Nice 1969}, Lect. Notes in Math 132, Springer-Verlag, New York 1970, 159-180;

\bibitem{R} R. T. Rockafellar, R.Wets, {\it Variational Analysis}, Springer Verlag 1998;

\bibitem{F} J. Franklin, {\it Methods of Mathematical Economics}, Springer Verlag 1980.
\end{thebibliography}
\end{document}